\documentclass[11pt]{article}
\usepackage[top=1in, bottom=1in, left=1in, right=1in]{geometry}
\usepackage{graphicx}
\usepackage{multicol}
\usepackage{times}
\usepackage{setspace} 
\usepackage{here}
\usepackage{amsfonts, amssymb}
\usepackage{mathrsfs}
\usepackage{amsthm}
\usepackage{amsmath}
\usepackage{bbm}
\numberwithin{equation}{section}
\newtheorem{theorem}{Theorem}[section]

\newtheorem{lemma}{Lemma}[section]
\theoremstyle{definition}
\newtheorem{dy}{Definition}[section]
\newtheorem{remark}{Remark}[section]
\usepackage{amsfonts, mathtools}
\usepackage{amsmath,amssymb}
\usepackage{hyperref}
\hypersetup{%
	colorlinks=true,
	linkcolor=black,
	citecolor=black, 
	urlcolor=blue
}
\usepackage{bxtexlogo}

\usepackage{listings}
\usepackage[edges]{forest}
\newdimen\forestdimen

\title{Monge solutions and uniqueness in multi-marginal optimal transport with hierarchical jumps}
\author{Zijian Xu}
\date{}
\begin{document}
	\maketitle
	\begin{abstract}
		We introduce Hierarchical Jump multi-marginal transport (HJMOT), a generalization of multi-marginal optimal transport where mass can ``jump" over intermediate spaces via augmented isolated points. Established on Polish spaces, the framework guarantees the existence of Kantorovich solutions and, under sequential differentiability and a twist condition, the existence and uniqueness of Monge solutions. This core theory extends robustly to diverse settings, including smooth Riemannian  manifolds, demonstrating its versatility as a unified framework for optimal transport across complex geometries.
	\end{abstract}
	
	\section{Introduction}
	\label{sec1}
	The theory of optimal transport (OT) fundamentally addresses the problem of efficiently moving mass between two spaces $X_0$ and $X_K$, minimizing a cost defined on pairs of points from these spaces (see \cite{MR2459454}). The framework of multi-marginal optimal (MOT) theory extends this by introducing intermediate spaces $X_1,\cdots,X_{K-1}$. It requires the mass to traverse a sequential path $X_0 \to X_1 \to \cdots \to X_{K-1} \to X_K$, minimizing a cumulative cost that is aggregated over each adjacent pair of spaces. However, this classical formulation imposes a rigid sequential structure, which may not be optimal in applications where direct jumps between non-adjacent spaces can lead to lower total cost. To model this flexibility, we augment each intermediate space $X_k$ ($1 \le k \le K-1$) with an isolated point $\partial_k$, defining an extended Polish space $\widehat{X}_k = X_k \cup {\partial_k}$. The  origin space and terminal space remain unchanged, i.e. $\widehat{X}_0 := X_0$ and $\widehat{X}_K := X_K$. A path traversing $\partial_k$ signifies that space $X_k$  is skipped in the transport leg between its neighbors. The active indices map $I(\omega)$ extracts the ordered sequence of spaces actually visited (excluding those skipped via 	$\partial_k$) from any path $\displaystyle \omega \in\Omega:=  \prod_{k=0}^{K}\widehat{X}_k$. We leverage this structure in the following way: We define a path cost function $c(\omega)$, which has a form of sums of  the pairwise costs between consecutive active indices identified by $I(\omega)$. The maximum adjacent cost $\tilde{c}_{i,i+1}$ provides an upper bound facilitating analysis. For fix $\mu_0 \in \mathcal{P}(X_0)$ and $\mu_K \in \mathcal{P}(X_k)$, the core Kantorovich formulation seeks a coupling $\pi \in \Pi(\mu_0,\cdots,\mu_K)$ minimizing the expected path cost :
	\begin{equation}
	M:=\inf_{ \substack{ \mu_k\in \mathcal{P}(\widehat{X}_k)\text{ for } k=1,\cdots,\,K-1 \\ \pi \in \Pi(\mu_0,\cdots,\,\mu_K)} } \int_{\Omega}c \ {\rm d}\pi.
	\end{equation}
	 In Section~\ref{sec2}, we establish the foundational theory of HJMOT on Polish spaces. Under a coercivity condition (Definition~\ref{Coercive}) that prevents mass from escaping to infinity, we prove the finiteness of the minimum cost $M$ (Theorem~\ref{Mfinity}) and the existence of a Kantorovich solution $\{\pi^*,\{\mu_k^*\}_{k=1}^{K-1}\}$ (Theorem~\ref{KP}). The proofs rely on the tightness of minimizing sequences and Prokhorov's theorem, demonstrating the robustness of the framework as they depend solely on the Polish space structure and the lower semicontinuity of the cost with coercivity. 
	 \par Building upon this, Section~\ref{sec3} addresses the existence and uniqueness of Monge solutions (deterministic optimal transport maps). We introduce several key concepts: the sequential differentiability (Definition~\ref{sd}) of the cost $c(\omega)$, the sequential twistedness condition (Definition~\ref{st}) requiring injectivity of a certain derivative mapping on optimal support, a strong coercivity condition ensuring compactness of sublevel sets of the cost, and a local control condition (Definition~\ref{lcc}) governing the local behavior of the cost near optimal paths. Under these conditions, Theorem~\ref{MP} guarantees a unique Monge solution $T^*: X_0 \to \Omega$, which explicitly decomposes to show which intermediate spaces are visited ($T^*_k(\widehat{x}_0) \in \widehat{X}_k$) and which are skipped ($T^*_k(\widehat{x}_0) = \partial_k$).
	 \par The core theory developed in Sections~\ref{sec2} and \ref{sec3} is notably versatile, relying only on the underlying Polish space structure. In Section~\ref{sec4}, we concretely realize this abstract framework on smooth Riemannian manifolds. Here, the sequential differentiability and twistedness conditions are naturally adapted to the geometric setting using directional derivatives along geodesics (Definitions~\ref{RKP} and~\ref{RMP}).
	\par The classical Monge problem with quadratic cost has been extensively studied in various geometric settings. In the Euclidean space, Brenier's seminal theorem established that when the source measure is absolutely continuous with respect to the Lebesgue measure, there exists a unique optimal transport map for the squared distance cost \cite{MR1100809}. This breakthrough was later extended to compact Riemannian manifolds by McCann \cite{MR1844080}, who proved existence and uniqueness of optimal maps under the same absolute continuity condition. Subsequently, Figalli \cite{MR2459454} resolved the non-compact case, showing that the problem admits a solution on general complete non-compact Riemannian manifolds.
	\par In Alexandrov spaces with lower curvature bounds, Bertrand \cite{MR2442054} proved the existence and uniqueness of optimal maps for the quadratic cost under the assumption that the source measure is absolutely continuous with respect to the Hausdorff measure, provided the space is finite-dimensional and compact. Although this result represents an important milestone, the Monge problem in fully general Alexandrov spaces remains only partially understood. Beyond the Riemannian and Alexandrov settings, significant advances have been made in more general metric measure spaces. Gigli \cite{MR2984123} obtained analogous results in synthetic Ricci curvature lower bound spaces (${\sf RCD}(K,N)$). Figalli and Rifford \cite{MR2647137} resolved the existence and uniqueness of Monge solutions on sub-Riemannian manifolds for the squared distance cost under the assumption of absolute continuity of the source measure. The present work situates itself within this broad landscape. The proposed framework for the HJMOT problem not only encompasses the aforementioned results of Brenier, McCann, Figalli, and Bertrand, but also aligns with recent developments in general metric measure spaces. In this sense, our results provide a unified perspective that integrates and extends the classical theory of optimal transport across Euclidean, Riemannian, Alexandrov, and synthetic curvature settings.
	\section{Multi-marginal optimal transport with hierarchical jumps.}
	\label{sec2}
	For $K\geq 1$ and $k \in \{0,1,2,\cdots,K\}$, let $X_k$ be a Polish space. First, we extend the  intermediate spaces  by adding an isolated point :
	\[
	\widehat{X}_k:= X_k \cup \{\partial_k\}\quad\text{for}\quad1\leq k \leq K-1.
	\]
	We topologize $\widehat{X}_k$ so that  $\{\partial_k\}$ is  isolated in $\widehat{X}_k$. Consequently, $\widehat{X}_k$ is  a Polish space (see \cite[Section~3]{MR1321597}).\\
	As origin space and terminal space, we shall set $\widehat{X}_0:=X_0$ and $\widehat{X}_K:=X_K$. Fix two Borel probability measures $\mu_0  \in \mathcal{P}(\widehat{X}_0)$ and $\mu_K \in \mathcal{P}(\widehat{X}_K)$, where $\mathcal{P}(\widehat{X}_k)$ denotes the set of all Borel probability measures on $\widehat{X}_k$. We introduce intermediate measures as follows:
	\[
	\mu_k \in \mathcal{P}(\widehat{X}_k) \quad\text{for}\quad 1 \leq k \leq K-1,
	\]
	which will be chosen optimally by the transport problem. We define the path space  as
	\[
	\Omega := \prod_{k=0}^{K}\widehat{X}_k=X_0 \times \left(\prod_{k=1}^{K-1} \widehat{X}_k\right) \times X_K,
	\]
	with product topology. A coupling is a Borel probability measure 
	\[
	\pi \in \mathcal{P}(\Omega)
	\]
	satisfying the following marginal constraints, i.e.
	\[
	{T_k}_\# \pi = \mu_k \quad\text{for}\quad k\in\{ 0,\cdots,K \},
	\]
	where $T_k : \Omega \mapsto \widehat{X}_k$ is the projection map defined by $T_k(\widehat{x}_0,\cdots \widehat{x}_K)=\widehat{x}_k$. For $2 \leq n \leq K+1$, define the set of admissible index sequences of length $n$ as:
	\[
	\Delta_{n}:=\{(k_0,k_1,\cdots,k_{n-1}) \in \mathbb{N}^n\mid  0=k_0<k_1<\cdots<k_{n-1}=K\}.
	\]
	The codomain of the active indices mapping $\mathcal{I}$ is the disjoint union of all such sets as follows:
	\[
	\mathcal{I}:=\bigsqcup_{n=2}^{K+1}\Delta_n.
	\]
	For $\omega=(\widehat{x}_0,\widehat{x}_1,\cdots,\widehat{x}_{K-1},\widehat{x}_K)\in\Omega$, we first define the number of actively visited stages as
	\[
	n(\omega) := \#\{1 \leq k \leq K-1 \mid \widehat{x}_k \in X_k\} + 2,
	\]
	and denote the ordered indices of visited intermediate spaces by
	\[
	\{k_1,k_2,\cdots,k_{n-2}\} := \{k \in \{1,2,\cdots,K-1\} \mid \widehat{x}_k \in X_k\},
	\]
	with the natural ordering $k_1 < k_2 < \cdots < k_{n-2}$. The active indices mapping $I: \Omega \to \mathcal{I}$ is then defined as
	\[
	I(\omega) := (n(\omega); (0,k_1,k_2,\cdots,k_{n-2},K)) \in \{n(\omega)\} \times \Delta_{n(\omega)} \subset \mathcal{I}.
	\]
	This mapping encodes the sequence of stage indices that the path $\omega$ actually traverses, always including the initial stage $0$ and terminal stage $K$. We define the path extraction map $\displaystyle\phi: \mathcal{I} \to \bigcup_{n=2}^{K+1} \left( \prod_{j=0}^{n-1} X_{k_j} \right)$ as \\
	\[
	\phi(n(\omega);(k_0,k_1,\cdots,k_{n-1})):=(x_{k_0},x_{k+1},\cdots,x_{k_{n-1}}),
	\]
	and we set $\Psi(\omega) := \phi(I(\omega))$. Given a family of lower semicontinuous pairwise cost functions 
	\[
	\{c_{i,j} :X_i \times X_j \to [0,+\infty] \  | \ 0\leq i<j\leq K   \},
	\]
	we define the path cost as
	\[
	c(\omega):=
	\sum_{l=0}^{n-2}c_{k_l,k_{l+1}}(x_{k_l},x_{k_{l+1}})\quad\text{for}\quad\Psi(\omega)=(x_{k_0},
	\cdots,x_{k_{n-1}}).
	\]
	For each $i=0,\cdots,K-1$, with $\widehat{x}_i\in \widehat{X}_i$ and $\widehat{x}_{i+1} \in \widehat{X}_{i+1}$, we define the maximum adjacent cost by
	\[
	\tilde{c}_{i,i+1}(\widehat{x}_i,\widehat{x}_{i+1}):=\left\{
	\begin{aligned}
		&c_{i,i+1}(\widehat{x}_i,\widehat{x}_{i+1}),&\quad&\text{if}\quad\widehat{x}_i\in X_i \quad\text{and}\quad\widehat{x}_{i+1} \in X_{i+1},\\
		&\sup_{\substack{i+1 < j \leq K\\ \widehat{x}_j\in X_j}}c_{i,j}(\widehat{x}_i,\widehat{x}_j),&\quad&\text{if}\quad\widehat{x}_i \in X_i \quad\text{and}\quad \widehat{x}_{i+1} =\partial_{i+1},\\
		&0,&\quad&\text{otherwise}.
	\end{aligned}
	\right.
	\]
	It follows that $\displaystyle c(\omega)\leq\sum_{i=0}^{K-1}\tilde{c}_{i,i+1}(\widehat{x}_i,\widehat{x}_{i+1})$ for every $\omega \in \Omega$. We fix $\mu_0\in\mathcal{P}(X_0)$ and $\mu_K \in \mathcal{P}(X_K)$. We seek to solve the following  optimization problem 
	
	\begin{equation}
		\label{KS}
		M:=\inf_{ \substack{ \mu_k\in \mathcal{P}(\widehat{X}_k)\text{ for }k=1,\cdots,K-1 \\ \pi \in \Pi(\mu_0,\cdots,\,\mu_K)} } \int_{\Omega} c(\omega) \ \pi({\rm d} \omega),
	\end{equation}
	where $\Pi(\mu_0,\cdots,\mu_K)=\{ \pi \in \mathcal{P}(\Omega) \mid {T_k}_\# \pi = \mu_k, \text{ for }k\in\{0,\cdots,K\}  \}$.
	In classical multi-marginal optimal transport (MOT), the intermediate marginals are fixed in advance, so the admissible coupling set is automatically tight (see Lemma~\ref{oldtight} below). In contrast, in our HJMOT setting, the intermediate marginals may vary with the sequence, and without additional control the mass could escape to infinity, preventing compactness. To rule out such escape and to guarantee tightness, we introduce the following condition.
	\begin{dy}[Coercivity]
		\label{Coercive}
	We say that cost function $c$ is \textit{coercive} if there exists a measurable function $\Phi: \Omega \to [0, +\infty]$  such that 
		\begin{itemize}
			\item[$(1)$] For every $R > 0$, the sublevel set $\{\omega \in \Omega : \Phi(\omega) \leq R\}$ is compact;
			\item[$(2)$] There exists a minimizing sequence $\{\pi_n\}$ of (\ref{KS}) such that $\displaystyle\sup_n \int_\Omega \Phi \, d\pi_n < \infty.$
		\end{itemize}
	\end{dy}
	\begin{remark}[Reduction to standard MOT]
		When no isolated points are introduced in the intermediate spaces (i.e., $\widehat{X}_k = X_k$ for all $k=1,\cdots,K-1$), the framework of HJMOT reduces to the standard multi-marginal optimal transport (MOT) problem. In this case, the active index map $I(\omega)$ necessarily satisfies $I(\omega) = \{0,1,\cdots,K\}$ for all paths $\omega \in \Omega$, forcing sequential transport through each intermediate space. The cost function $c(\omega)$ is then simplified to the classical MOT form: 
		\[
		c(\omega) := \sum_{i=0}^{K-1} c_{i,i+1}(x_i, x_{i+1}).
		\]
		Thus, HJMOT is a strict generalization of MOT, allowing mass to bypass certain intermediate spaces via hierarchical jumps
	\end{remark}
	\begin{theorem}
		\label{Mfinity}
		Fix $\mu_0 \in \mathcal{P}(X_0)$ and $\mu_K \in \mathcal{P}(X_K)$. If there exist intermediate measures   $\mu_k \in \mathcal{P}(\widehat{X}_k)$ for $k=1,\cdots,K-1$ such that
		\[
		\sum_{i=0}^{K-1}\tilde{C}_{i,i+1}(\mu_{i},\mu_{i+1}) < \infty,
		\] 
		where $\displaystyle\tilde{C}_{i,i+1}(\mu_i,\mu_{i+1}):=\inf_{\pi_i\in\Pi(\mu_i,\mu_{i+1})} \int_{\widehat{X}_i \times \widehat{X}_{i+1}} \tilde{c}_{i,i+1}(\widehat{x}_i,\widehat{x}_{i+1})\pi_i({\rm d}\widehat{x}_i,{\rm d} \widehat{x}_{i+1})$, then the optimal value $M$ in  (2.1) is finite.
	\end{theorem}
	\begin{proof}
		For $i=0,\cdots,K-1$ and every $\varepsilon_i>0$, there exists $\pi_i\in\Pi(\mu_i,\mu_{i+1})$ such that
		\[
		\int_{\widehat{X}_i \times \widehat{X}_{i+1}} \tilde{c}_{i,i+1}(\widehat{x}_i,\widehat{x}_{i+1})\ \pi_i({\rm d}\widehat{x}_i,{\rm d}\widehat{x}_{i+1}) \leq \tilde{C}_{i,i+1}(\mu_i,\mu_{i+1})+\varepsilon_i.
		\]
		Take $\pi_0\in\Pi(\mu_0,\mu_1)$ on $\widehat{X}_0 \times \widehat{X}_1$ and $\pi_1 \in \Pi(\mu_1,\mu_2)$ on $\widehat{X}_1\times \widehat{X}_2$. By the disintegration theorem, we disintegrate $\pi_1$ with respect to its $\widehat{X}_1$–marginal $\mu_1$, obtaining
		\[
		\pi_1({\rm d}\widehat{x}_1,{\rm d}\widehat{x}_2)=\mu_1({\rm d}\widehat{x}_1)\otimes \kappa_1({\rm d}\widehat{x}_2\mid\widehat{x}_1),
		\]
		where $\kappa_i(\cdot\mid\widehat{x}_i)$  is the conditional probability measure of $\widehat{x}_{i+1}$  given $\widehat{x}_i$. Next, combine $\pi_0$ and $\pi_1$ to construct $\pi_{01} \in \mathcal{P}(\widehat{X}_0\times \widehat{X}_1\times \widehat{X}_2)$ defined by
		\[
		\pi_{01}({\rm d}\widehat{x}_0,{\rm d}\widehat{x}_1,{\rm d}\widehat{x}_2):= \pi_0({\rm d}\widehat{x}_0,{\rm d}\widehat{x}_1)\otimes \kappa_1({\rm d}\widehat{x}_2\mid\widehat{x}_1).
		\]
		Now assume we have constructed $\pi_{0\cdots i}({\rm d}\widehat{x}_0,\cdots,{\rm d}\widehat{x}_{i+1})=\pi_{0\cdots (i-1)}({\rm d}\widehat{x}_0,\cdots,{\rm d}\widehat{x}_i)\otimes \kappa_i({\rm d}\widehat{x}_{i+1}\mid{\rm d}\widehat{x}_i) \in \mathcal{P}(\prod_{j=0}^{i+1}(\widehat{X}_j))$ with pairwise marginals $(T_j,T_{j+1})_\#\pi_{0\cdots i}=\pi_j$ for $j=0,\cdots,i$ and ${T_k}_\# \pi_{0\cdots i} = \mu_k$ for $k=0,\cdots,i+1$. Then we take $\pi_{i+1}\in \Pi(\mu_{i+1},\mu_{i+2})$ and disintegrate $\pi_{i+1}$ with respect to $\mu_{i+1}$ as
		\[
		\pi_{i+1}({\rm d}\widehat{x}_{i+1},{\rm d}\widehat{x}_{i+2})=\mu_{i+1}({\rm d}\widehat{x}_{i+1}) \otimes \kappa_{i+1}({\rm d}\widehat{x}_{i+2}\mid\widehat{x}_{i+1}).
		\]
		We can now extend $\pi_{0\cdots i}$ to $\pi_{0\cdots (i+1)} \in \mathcal{P}(\prod_{j=0}^{i+2}(\widehat{X}_j))$ by
		\[
		\pi_{0\cdots (i+1)}({\rm d}\widehat{x}_0,\cdots {\rm d}\widehat{x}_{i+2})=\pi_{0\cdots i}({\rm d}\widehat{x}_0,\cdots,{\rm d}\widehat{x}_{i+1}) \otimes \kappa_{i+1}({\rm d}\widehat{x}_{i+2}\mid{\rm d}\widehat{x}_{i+1}).
		\]
		For each $k=0,\cdots,i+1$ and any Borel set $A_k\subset \widehat{X}_k$, 
		\[
		\begin{aligned}
			{T_k}_\#\pi_{0\cdots (i+1)}(A_k)&=\int_{\prod_{j=0}^{i+2} \widehat{X}_j} \mathbbm{1}_{A_k}(\widehat{x}_k) \pi_{0\cdots (i+1)}({\rm d}\widehat{x}_0,\cdots,{\rm d}\widehat{x}_{i+2})\\
			&=\int_{\prod_{j=0}^{i+2} \widehat{X}_j} \mathbbm{1}_{A_k}(\widehat{x}_k) \pi_{0\cdots i}({\rm d}\widehat{x}_0,\cdots,{\rm d}\widehat{x}_{i+1})\otimes \kappa_{i+1}({\rm d}\widehat{x}_{i+2}\mid \widehat{x}_{i+1})\\
			&=\int_{\prod_{j=0}^{i+1} \widehat{X}_j} \mathbbm{1}_{A_k}(\widehat{x}_k) \pi_{0\cdots i}({\rm d}\widehat{x}_0,\cdots,{\rm d}\widehat{x}_{i+1}) \cdot \int_{\widehat{X}_{i+2}}  \kappa_{i+1}({\rm d}\widehat{x}_{i+2}\mid \widehat{x}_{i+1})\\
			&={T_k}_\#\pi_{0\cdots i}(A_k) \times 1=\mu_k(A_k).
		\end{aligned}
		\]
		Thus, ${T_k}_\# \pi_{0\cdots (i+1)}=\mu_k$ for $k=0,\cdots,i+1$. For each $k=i+2$, and any Borel set $A_{i+2}\subset \widehat{X}_{i+2}$ 
		\[
		\begin{aligned}
			{T_{i+1}}_\#\pi_{0\cdots (i+2)}(A_{i+2})&=\int_{\prod_{j=0}^{i+2} \widehat{X}_{j}} \mathbbm{1}_{A_{i+2}}(\widehat{x}_{i+2}) \pi_{0\cdots (i+1)}({\rm d}\widehat{x}_0,\cdots,{\rm d}\widehat{x}_{i+2})\\
			&=\int_{\prod_{j=0}^{i+2} \widehat{X}_j} \mathbbm{1}_{A_{i+2}}(\widehat{x}_{i+2}) \pi_{0\cdots i}({\rm d}\widehat{x}_0,\cdots,{\rm d}\widehat{x}_{i+1})\otimes \kappa_{i+1}({\rm d}\widehat{x}_{i+2}\mid \widehat{x}_{i+1})\\
			&=\int_{\prod_{j=0}^{i+1}\widehat{X}_j}    \pi_{0\cdots i}({\rm d}\widehat{x}_0,\cdots,{\rm d}\widehat{x}_{i+1})\int_{A_{i+2}} \kappa_{i+1}({\rm d}\widehat{x}_{i+2}\mid \widehat{x}_{i+1}) \\
			&=\int_{\widehat{X}_{i+1} \times A_{i+2}}\mu_{i+1}({\rm d}\widehat{x}_{i+1})\otimes\kappa_{i+1}({\rm d}\widehat{x}_{i+2}\mid \widehat{x}_{i+1})\\
			&=\int_{\widehat{X}_{i+1} \times A_{i+2}} \pi_{i+1}({\rm d}\widehat{x}_{i+1},{\rm d}\widehat{x}_{i+2})=\pi_{i+1}\left( \widehat{X}_{i+1}\times A_{i+2}\right)=\mu_{i+2}(A_{i+2}).\\
		\end{aligned}
		\]
		Thus, ${T_k}_\# \pi_{0\cdots (i+1)}=\mu_k$ for $k=0,\cdots,i+2$. For each $j=0,\cdots,i$ and any Borel set $A_j\subset\widehat{X}_j$, $A_{j+1} \subset \widehat{X}_{j+1}$, 
		\[
		\begin{aligned}
			(T_j,T_{j+1})_\#\pi_{0\cdots j}(A_j,A_{j+1})&=\int_{\prod_{k=0}^{j+1}\widehat{X}_k} \mathbbm{1}_{A_j}(\widehat{x}_j)\mathbbm{1}_{A_{j+1}}(\widehat{x}_{j+1}) \pi_{0\cdots j}({\rm d}\widehat{x}_0,\cdots,{\rm d}\widehat{x}_{j+1})\\
			&=\int_{\widehat{X}_0}\cdots\int_{\widehat{X}_j}\mathbbm{1}_{A_j}(\widehat{x}_j)\left\{\int_{\widehat{X}_{j+1}} \mathbbm{1}_{A_{j+1}}(\widehat{x}_{j+1})\kappa_j({\rm d}\widehat{x}_{j+1}\mid\widehat{x}_j)\right\} \pi_{0\cdots (j-1)}({\rm d}\widehat{x}_0,\cdots,{\rm d}\widehat{x}_j)\\
			&=\int_{\widehat{X}_0}\cdots\int_{\widehat{X}_j}\mathbbm{1}_{A_j}(\widehat{x}_j) \kappa_j(A_{j+1}\mid\widehat{x}_j) \pi_{0\cdots (j-1)}({\rm d}\widehat{x}_0,\cdots,{\rm d}\widehat{x}_j)\\
			&=\int_{A_j}\int_{A_{j+1}}\mu_j({\rm d}\widehat{x}_j) \kappa_{j}({\rm d}\widehat{x}_{j+1}\mid x_j)\\
			&=\int_{A_j}\int_{A_{j+1}} \pi_j({\rm d}\widehat{x}_j,{\rm d}\widehat{x}_{j+1})=\pi_j(A_j \times A_{j+1}).
		\end{aligned}
		\]
		Thus, $(T_j,T_{j+1})_\#\pi_{0\cdots j}=	\pi_j$ for $j=0,\cdots,i$. For each $j = 0,\cdots,i$, extending $\pi_{0\cdots j}$ to $\pi_{0\cdots(j+1)}$ does not alter the previously established pairwise marginal $(T_j,T_{j+1})_\#\pi_{0\cdots j} = \pi_j$. Indeed, for any Borel sets $A_j \subset \widehat{X}_j$ and $A_{j+1} \subset \widehat{X}_{j+1}$, we have 
		\[
		\begin{aligned}
			(T_j,T_{j+1})_\# \pi_{0\cdots(j+1)}(A_j \times A_{j+1})
			&= \int_{\prod_{k=0}^{j+1}\widehat{X}_k} \mathbbm{1}_{A_j}(\widehat{x}_j)\,\mathbbm{1}_{A_{j+1}}(\widehat{x}_{j+1})\,\pi_{0\cdots(j+1)}({\rm d}\widehat{x}_0,\cdots,{\rm d}\widehat{x}_{j+1})\\
			&= \int_{\widehat{X}_0}\cdots\int_{\widehat{X}_{j+1}} \mathbbm{1}_{A_j}(\widehat{x}_j)\,\mathbbm{1}_{A_{j+1}}(\widehat{x}_{j+1})\,\pi_{0\cdots j}({\rm d}\widehat{x}_0,\cdots,{\rm d}\widehat{x}_{j+1})\\
			&= \pi_j(A_j \times A_{j+1})\,.
		\end{aligned}
		\]
		Therefore, after extending to $j+1$, we still have $(T_j,T_{j+1})_\#\pi_{0\cdots(j+1)} = \pi_j$. In summary, after iterating up to $i=K-1$, we construct $\pi:=\pi_{0\cdots K} \in \mathcal{P}(\Omega)$ with $(T_j,T_{j+1})_\#\pi = \pi_j$ for all $0 \le j < K$ and ${T_k}_\#\pi = \mu_k$ for all $0 \le k \le K$. By taking $i=K-1$, we obtain the global coupling  $\pi:=\pi_{0\cdots K-1} \in \mathcal{P}(\Omega)$. Note that each $\tilde{c}_{i,i+1}(\widehat{x}_i,\widehat{x}_{i+1})$ depends only on the pair $(\widehat{x}_i,\widehat{x}_{i+1})$ and $(T_i,T_{i+1})_\#\pi=\pi_i$. Therefore, using  $\displaystyle c(\omega)\leq\sum_{i=0}^{K-1}\tilde{c}_{i,i+1}(\widehat{x}_i,\widehat{x}_{i+1})$, we have
		\[
		\begin{aligned}
			M\leq \int_{\Omega}c(\omega)\pi ({\rm d}\omega)&=\sum_{i=0}^{K-1}\int_{\widehat{X}_i\times \widehat{X}_{i+1}} \tilde{c}_{i,i+1}(\widehat{x}_i,\widehat{x}_{i+1})\pi_i({\rm d}\widehat{x}_i,{\rm d}\widehat{x}_{i+1})\\
			&\leq \sum_{i=0}^{K-1} \left(\tilde{C}_{i,i+1}(\mu_i,\mu_{i+1})+\varepsilon_i \right)<\infty,
		\end{aligned} 
		\]
		which shows $M<\infty$.
	\end{proof}
	\begin{lemma}
		\label{lsc}
		$c: \Omega \to [0,+\infty]$ is lower semicontinuous.
	\end{lemma}
	\begin{proof}
		For $\omega \in \Omega$, take $\{\omega_n\} \subset \Omega$ such that $\omega_n \to \omega$ as $n \to \infty$. If $\displaystyle c(\omega)=\sum_{l=0}^{n-2}c_{k_l,k_{l+1}}(x_{k_l},x_{k_{l+1}})$ for $\Psi(\omega)=(x_{k_0},\cdots,x_{k_{n-1}})$ (e.g., $\omega =(\widehat{x}_0 ,\cdots,\partial_i,\cdots,\widehat{x}_K)$ for some $i=1,\cdots,K-1$), then for sufficiently large $n$, $\displaystyle c(\omega_n)= \sum_{l=0}^{n-2}c_{k_l,k_{l+1}}(x_{k_l}^n,x_{k_{l+1}}^n)$. Therefore
		\[
		\varliminf_{n\to \infty} c(\omega_n) =\varliminf_{n\to \infty} \sum_{l=0}^{n-2}c_{k_l,k_{l+1}}(x_{k_l}^n,x_{k_{l+1}}^n) \geq \sum_{l=0}^{n-2}  \varliminf_{n\to \infty} c_{k_l,k_{l+1}}(x_{k_l}^n,x_{k_{l+1}}^n)\geq \sum_{l=0}^{n-2}c_{k_l,k_{l+1}}(x_{k_l},x_{k_{l+1}})=c(\omega).
		\]
	\end{proof}
	\begin{lemma}
		\label{oldtight}
		If each $\widehat{X}_k$ $(k=0,\cdots,K)$ is a Polish space and fix $\mu_k \in \mathcal{P}(\widehat{X}_k)$, then the family of couplings $\Pi(\mu_0,\cdots,\mu_K)$ is uniformly tight.
	\end{lemma}
	\begin{proof}
		Since each $\widehat{X}_k$ is a Polish space, every probability measure $\mu_k \in \mathcal{P}(\widehat{X}_k)$ is tight, i.e., for every $\varepsilon>0$, there exists a compact set $A_k \subset X_k$ such that $\mu_k( \widehat{X}_k \setminus A_k) < \varepsilon$. Define a compact subset  $\widehat{A} \subset\Omega$ by
		\[
		\widehat{A}:= \prod_{k=0}^{K}  A_k ,
		\]
		choosing $\mu_k(\widehat{X}_k \setminus A_k) < \varepsilon/(K+1)$ for $k=0,\cdots,K$. Then, for every $\pi \in \Pi(\mu_0,\cdots,\mu_K)$, we have
		\[
		\begin{aligned}
			\pi(\Omega \setminus \widehat{A})&\leq \pi\left( T_0^{-1}(\widehat{X}_0\setminus A_0)  \cup  \cdots \cup T_K^{-1}(\widehat{X}_K\setminus A_K)\right)\\
			&\leq \sum_{k=0}^{K} \pi (T_k^{-1} (\widehat{X}_k \setminus A_k))\\
			&=\sum_{k=0}^{K} \mu_k(\widehat{X}_k \setminus A_k)\\
			&<\varepsilon,
		\end{aligned}
		\]
		showing that $\Pi(\mu_0,\dots,\mu_K)$ is uniformly tight.
	\end{proof}
	\begin{lemma}
		\label{newtight}
		If the  cost function $c$ is coercive, then there exists a minimizing sequence $\{\pi_n\}$ of (\ref{KS}) is uniformly tight.
	\end{lemma}
	\begin{proof}
		By the coercivity condition, there exists a measurable function $\Phi: \Omega \to [0, +\infty]$ and a constant $C > 0$ such that 
		\begin{itemize}
			\item[$(1)$] For every $R > 0$, the sublevel set $\{\omega \in \Omega\mid \Phi(\omega) \leq R\}$ is compact;
			\item[$(2)$] There exists a minimizing sequence $\{\pi_n\}$ of (\ref{KS}) such that $\displaystyle\sup_n \int_\Omega \Phi \, d\pi_n = C.$
		\end{itemize}
		For any $\varepsilon > 0$, set $\displaystyle R_\varepsilon = \frac{2C}{\varepsilon}$. Consider the set $K_\varepsilon = \{\omega \in \Omega \mid\Phi(\omega) \leq R_\varepsilon\}$, by condition $(1)$ , $K_\varepsilon$ is compact. For each $n$, we use Markov's inequality as
		\[
		\pi_n(\{\omega \in \Omega : \Phi(\omega) > R_\varepsilon\}) \leq \frac{1}{R_\varepsilon} \int_\Omega \Phi \, d\pi_n.
		\]
		Using condition (2), we have
		\[
		\pi_n(\{\omega \in \Omega : \Phi(\omega) > R_\varepsilon\}) \leq  \frac{C}{R_\varepsilon} = \frac{\varepsilon}{2}.
		\]
		Note that $\Omega \setminus K_\varepsilon = \{\omega \in \Omega \mid \Phi(\omega) > R_\varepsilon\}$, 
		so we have $\pi_n(\Omega \setminus K_\varepsilon) \leq \frac{\varepsilon}{2} \quad \text{for every } n$. Therefore
		\[
		\sup_n \pi_n(\Omega \setminus K_\varepsilon) < \varepsilon.
		\]
		Since $\varepsilon > 0$ was arbitrary, and for each $\varepsilon$ we have constructed a compact set $K_\varepsilon \subset \Omega$ such that $\sup_n \pi_n(\Omega \setminus K_\varepsilon) < \varepsilon$, it follows that $\{\pi_n\}$ is uniformly tight.
	\end{proof}
    \begin{theorem}[Kantorovich solution]
    	\label{KP}
    	If $M<\infty$ and $c$ is coercive, then the optimization problem (\ref{KS}) admits a minimizer. In particular, there exists an optimal coupling $\pi^* \in \Pi(\mu_0,\mu_1^*,\cdots,\mu_{K-1}^*,\mu_K)$ along with optimal intermediate measures $\mu_1^*,\cdots,\mu_{K-1}^*$ such that $\{\pi^*,\,\{\mu_k^*\}_{k=1}^{K-1}\}$ attains the infimum in (\ref{KS}).
    \end{theorem}
    \begin{proof}
    	Since $c$ is coercive, then there exists a minimizing sequence $\{\pi_n\} \subset  \Pi(\mu_0,\mu_1^n,\cdots,\mu_{K-1}^n,\mu_K)$ of (\ref{KS}) with $\mu_k^n={T_k}_\# \pi_n$ for $k=1,\cdots,K-1$ is uniformly tight. By definition of $M$, we have
    	\[
    	\lim_{n\to\infty} \int_{\Omega} c \ {\rm d}\pi_n = M.
    	\]
    	Since $\{\pi_n\} \subset \mathcal{P}(\Omega)$ is tight by Lemma~\ref{newtight}, Prokhorov’s theorem guarantees a weakly convergent subsequence  $\pi_{n_j} \rightharpoonup \pi^* \in \mathcal{P}(\Omega)$ as $j \to \infty$.  In particular, for every $f \in C_b(\Omega)$,
    	\[
    	\lim_{j\to\infty} \int_{\Omega} f\ {\rm d}\pi_{n_j} = \int_{\Omega} f \ {\rm d}\pi^*.
    	\]
    	For every $ g \in C_b(\widehat{X}_k)$, $g\circ T_k \in C_b(\Omega)$ and
    	\[
    		\lim_{j\to\infty}\int_{\widehat{X}_k} g \ {\rm d}({T_k}_\# \pi_{n_j})=\lim_{j\to\infty}\int_{\Omega} g \circ T_k \ {\rm d}\pi_{n_j}=\int_{\Omega} g \circ T_k \ {\rm d}\pi^*=\int_{\widehat{X}_k} g \ {\rm d}({T_k}_\# \pi^*).
    	\]
    	Applying this to $f = g\circ T_k$ for an arbitrary $g \in C_b(\widehat{X}_k)$ yields ${T_k}_\# \pi_{n_j} \rightharpoonup {T_k}_\# \pi^*$ as $j \to \infty$ , i.e. $\mu^{n_j}_k := {T_k}_\#\pi_{n_j}$ converges weakly to $\mu^*_k := {T_k}_\#\pi^*$ for each $k=1,\cdots,K-1$. In particular, $\pi^* \in \Pi(\mu_0,\mu^*_1,\cdots,\mu^*_{K-1},\mu_K)$ since the endpoint marginals are preserved. Because $c$ is lower semicontinuous, we have
    	\[
    		M \leq\int_{\Omega} c\ {\rm d}\pi^* \leq \varliminf_{j\to \infty} \int_{\Omega} c\ {\rm d}\pi_{n_j}=M.
    	\]
    	Therefore, $\{ \pi^*,\{\mu^*_k \}_{k=1}^{K-1} \}$ is a minimizer of (\ref{KS}).
    \end{proof}
    \begin{remark}[Robustness of proof]
    The proof of Theorem~\ref{KP} depends only on the Polish structure of the spaces, the lower semicontinuity of $c$, and Definition~\ref{Coercive}. It is independent of how jumps are modeled, so long as the resulting spaces remain Polish. This ensures structural stability, provides a unified existence result for Sections~\ref{sec4} without repeating arguments, and highlights the framework's modeling flexibility.
    \end{remark}
    \begin{remark}[On the coercivity condition and its verification]
    The coercivity condition (Definition~\ref{Coercive}) is essential for the existence proof in Theorem~\ref{KP}, though its verification depends on the problem structure. When HJMOT reduces to classical MOT or OT, coercivity becomes unnecessary as Lemma~\ref{oldtight} ensures uniform tightness of minimizing sequences. In compact Polish spaces, coercivity holds trivially. For non-compact but proper Polish spaces, coercivity is typically straightforward to verify, as one can readily construct suitable $\Phi$ functions that ensure both compact sublevel sets and bounded integrals along minimizing sequences. In contrast, for non-compact and non-proper Polish spaces, coercivity requires more careful analysis and construction of $\Phi$ tailored to the specific geometry of the spaces and the behavior of the cost function. Overall, coercivity represents a verifiable and often mild assumption in practical applications of the HJMOT framework.
    \end{remark}
    \section{Existence and Uniqueness for the Monge Problem}
    \label{sec3}
    \begin{dy}[Sequential differentiability]
    	\label{sd}
    		Let $(X_0,{\sf d}_0)$ be a Polish space. We say that the cost function $c: \Omega \to [0, +\infty]$ is \textit{sequentially differentiable} at $\widehat{x}_0 \in X_0$ if for any fixed path $\omega=(\widehat{x}_0,\cdots,\widehat{x}_K) \in \Omega$ and there exists sequence $\{\widehat{x}_0^n\}_{n=1}^\infty \subset X_0$ converging to $\widehat{x}_0$, the limit
    			\[
    			D_{\{\widehat{x}_0^n\},\,\widehat{x}_0} c(\omega) := \lim_{n \to \infty} \frac{c(\widehat{x}_0^n,\widehat{x}_1,\cdots,\widehat{x}_K) - c(\widehat{x}_0,\widehat{x}_1,\cdots\,\widehat{x}_K)}{{\sf d}_0(\widehat{x}_0^n,\widehat{x}_0)}
    			\]
    			always exists in $\mathbb{R}$. This limit depends on the choice of the sequence $\{\widehat{x}_0^n\}$.
    \end{dy}
    \begin{dy}[Cyclical monotonicity]
    	\label{cm}
    	A set $S \subset \Omega$ is called $c$-\textit{cyclically monotone} if for every finite collection of points $\{(\widehat{x}_0^i,\cdots,\widehat{x}_K^i)\}_{i=1}^m \subset S$ and for every choice of permutations $\sigma_0,\cdots,\sigma_K$ of the index set $\{1,\cdots,m\}$, one has 
    	\[
    	\sum_{i=1}^m c(\widehat{x}_0^i,\cdots,\widehat{x}_K^i) 
    	\;\le\; \sum_{i=1}^m c\Big(\widehat{x}_0^{\sigma_0(i)},\,\widehat{x}_1^{\sigma_1(i)},\,\cdots,\,\widehat{x}_K^{\sigma_K(i)}\Big)\,. 
    	\] 
    \end{dy}
    \begin{dy}[Sequential twistedness]
    	\label{st}
    	A cost function $c$ is called \textit{sequential twisted} 
    	if $c(\omega)$ is sequential differentiable at $\widehat{x}_0$ and the mapping
    	\[
    	(\widehat{x}_1,\cdots,\widehat{x}_K) \;\mapsto\; D_{\{\widehat{x}_0^n\} ,\,\widehat{x}_0} c(\widehat{x}_0,\widehat{x}_1,\cdots,\widehat{x}_K) 
    	\] 
    	is injective on the subset $S_{\widehat{x}_0}:=\{ \omega=(\widehat{x}_0,\widehat{x}_1,\cdots,\widehat{x}_K) \in S \mid \widehat{x}_i\in\widehat{X}_i \text{ for }1\leq i \leq K \}$, where $S$ is a $c$-cyclically monotone set.\\
    	In particular, if Polish spaces $X_0$ that lack a nontrivial curve structure (such as finite or countable discrete spaces, singletons, or totally disconnected perfect sets), the cost function $c$ is called \textit{discrete-type sequential twisted} if for $\mu_0$-a.e.  $\widehat{x}_0 \in X_0$, the set of optimal continuations
    	\[
    	\mathcal{F}(\widehat{x}_0)  := \arg\min_{(\widehat{x}_1,\cdots,\,\widehat{x}_K) \in \widehat{X}_1\times\cdots\times\widehat{X}_K}
    	c(\widehat{x}_0,\widehat{x}_1,\cdots,\widehat{x}_K)
    	\]
    	is a singleton.
    \end{dy}
     \begin{dy}[Strongly  coercivity]
     \label{sc}
     	Let $\{\pi^*,\{\mu_k^*\}_{k=1}^{K-1}\}$ be a solution of (\ref{KS}). We say that cost function $c$ is \textit{strongly coercive} if there exists a measurable function $\Phi: \Omega \to [0, +\infty]$  such that 
    	\begin{itemize}
    		\item[$(1)$] For every $R > 0$, the sublevel set $\{\omega \in \Omega : \Phi(\omega) \leq R\}$ is compact;
    		\item[$(2)$] There exists a minimizing sequence $\{\pi_n\}$ of (\ref{KS}) such that $\displaystyle\sup_n \int_\Omega \Phi \, d\pi_n < \infty.$
    		\item[$(3)$] There exists $\varepsilon>0$ such that for any $\mu_0$-a.e. $\widehat{x}_0 \in X_0$, $\{\omega \in {\rm supp}(\pi^*) \mid T_0(\omega)=\widehat{x}_0,c(\omega)\leq h(\widehat{x}_0)+\varepsilon\}$ is a compact set where $h(\widehat{x}_0):=\inf\{c(\omega)\mid \omega \in {\rm supp}(\pi^*), T_0(\omega)=\widehat{x}_0\}$.
    	\end{itemize}
    \end{dy}
    \begin{lemma}[Kantorovich duality \cite{MR4809472}]
    	\label{Kd}
    	Let $\{\pi^*,\{\mu_k^*\}_{k=1}^{K-1}\}$be an optimal solution to the optimization problem (\ref{KS}), then there exist Borel functions $\{v_k\}_{k=0}^K$ (called $c$-splitting functions)  such that 
    	\[
    	\sum_{k=0}^{K} v_k(\widehat{x}_k) \leq c(\omega)\quad\text{for all}\quad\omega(\widehat{x}_0,\cdots,\widehat{x}_K) \in \Omega
    	\]
    	and
    	\[
    	\sum_{k=0}^{K} v_k(\widehat{x}_k) =c(\omega)\quad\text{for}\quad\omega \in{\rm supp}(\pi^*).
    	\]
    \end{lemma}
    \begin{proof}
    	Let $\{\pi^*,\{\mu_k^*\}_{k=0}^K\}$ be an optimal solution to the optimization problem (\ref{KS}) where we set $\mu_0^*:=\mu_0$ and $\mu_K^*:=\mu_K$. Set 
    	\[
    	M=\inf_{\pi \in \Pi(\mu_0^*,\cdots,\mu_K^*)} \int_\Omega c(\omega) \ \pi({\rm d}\omega)
    	\]
    	and
    	\[
    	D:=\sup\left\{\left. \sum_{k=0}^{K}\int_{\widehat{X}_k} v_k \ {\rm d}\mu_k^*\ \right|\ v_k \in L^1(\widehat{X}_k;\mu_k^*),\,\sum_{k=0}^{K}v_k(x_k) \leq c(\omega)\text{ for all }\omega \in \Omega    \right\}.
    	\]
    	Since $c$ is non-negative semicountinuous and $\Pi(\mu_0^*,\cdots,\mu_K^*)$ is uniformly tight (Lemma~\ref{oldtight}), then by Kellerer's theorem (\cite{MR761565}) we have $D=M$ and there exists $v_k \in L^1(\widehat{X}_k,\mu_k^*)$ such that for evey $\omega \in \Omega$,
    	\[
    	\sum_{k=0}^{K} v_k (\widehat{x}_k)\leq c(\omega)\quad\text{and}\quad \sum_{k=0}^{K} \int_{\widehat{X}_k} v_k\ {\rm d}\mu_k^* =M.
    	\]
    	Since 
    		\[
    	\int_\Omega \left( c - \sum_{k=0}^{K} v_k \right)\ {\rm d}\pi^*=\int_\Omega c \ {\rm d}\pi^* - \sum_{k=0}^{K}\int_{\widehat{X}_k} v_k\ {\rm d}\mu_k^*=M-M=0.
    	\] 
    	Therefore $\displaystyle \sum_{k=0}^{K} v_k(\widehat{x}_k) =c(\omega)\text{ for }\pi^*\text{-a.e. } \omega \in \Omega$. Suppose there exists $\omega \in \text{supp}(\pi^*)$ and an open neighbourhood $ U$ of $ \omega$ such that for any  $\omega' \in U$, we have $\displaystyle \sum_{k=0}^{K} v_k(\widehat{x}'_k) <c(\omega')$. Since $\omega \in \text{supp}(\pi^*)$, we can see that $\pi^*(U)>0$. Then
    	\[
    	\int_\Omega \left( c - \sum_{k=0}^{K} v_k \right)\ {\rm d}\pi^* \geq \int_U \left( c - \sum_{k=0}^{K} v_k \right)\ {\rm d}\pi^* >0,
    	\]
    	this contradicts the statement that$\displaystyle \sum_{k=0}^{K} v_k(\widehat{x}_k) =c(\omega)\text{ for }\pi^*\text{-a.e. } \omega \in \Omega$. Therefore, $\displaystyle \sum_{k=0}^{K} v_k(\widehat{x}_k) =c(\omega)$ for $\omega \in {\rm supp}(\pi^*)$.
    \end{proof}
    \begin{lemma}[cf. \cite{MR3438523}]
    	\label{lcm}
    	Let $\{\pi^*,\{\mu_k^*\}_{k=1}^{K-1}\}$ be a solution of (\ref{KS}). Then ${\rm supp}(\pi^*)$ is a $c$-cyclically monotone set.
    \end{lemma}
    \begin{proof}
    	By Lemma~\ref{Kd}, there exist Borel functions $\{v_k\}_{k=0}^K$ satisfying:
    	\[
    	\begin{aligned}
    		\sum_{k=0}^K v_k(\widehat{x}_k) &\leq c(\omega) \quad \forall \omega = (\widehat{x}_0, \cdots, \widehat{x}_K) \in \Omega, \\
    		\sum_{k=0}^K v_k(\widehat{x}_k) &= c(\omega) \quad \forall \omega \in {\rm supp}(\pi^*).
    	\end{aligned}
    	\]
    	Let $S = \{\omega_1, \cdots, \omega_m\} \subset {\rm supp}(\pi^*)$ be an arbitrary finite collection of points, where $\omega_i = (\widehat{x}_0^i, \cdots, \widehat{x}_K^i)$ for $ i =1,\cdots, m$. Define a discrete probability measure $\beta$ supported on $S$ as
    	\[
    	\beta = \frac{1}{m} \sum_{i=1}^m \delta_{\omega_i}.
    	\]
    	where $\delta_{\omega}$ is a Dirac measure. Let $\alpha$ be any probability measure on $\Omega$ with identical marginals to $\beta$, i.e., 
    	\[
    		(T_k)_\# \alpha = (T_k)_\# \beta \quad\text{for every}\quad k=0,\cdots,K. 
    	\]
    	Consider the cost functional evaluated at $\alpha$,
    	\[
    	\begin{aligned}
    		\int_\Omega c\ {\rm d}\beta &=\frac{1}{m}\sum_{i=0}^{K}\sum_{k=0}^{K}v_k(T_k(\omega_i))\\
    		&=\frac{1}{m}\sum_{i=1}^{m}\sum_{k=0}^{K}\int_{_{\widehat{X}_k}} v_k\ {\rm d}({T_k}_\# \delta_{\omega_i})\\
    		&=\sum_{k=0}^{K}\int_{\widehat{X}_k} v_k\ {\rm d}({T_k}_\#\beta)\\
    		&=\sum_{k=0}^{K}\int_{\widehat{X}_k} v_k\ {\rm d}({T_k}_\#\alpha)=\int_\Omega \sum_{k=0}^{K} v_k(T_k(\omega)) \ {\rm d}\alpha(\omega)\leq\int_\Omega c\ {\rm d}\alpha.
    	\end{aligned}
    	\]
       	Thus we have established $\int_\Omega c\   {\rm d}\beta \leq \int_\Omega c\   {\rm d}\alpha$ for any measure $\alpha$ with the same marginals as $\beta$.\\
    	To verify $c$-cyclical monotonicity, consider any permutation $\sigma_0,\cdots,\sigma_K$  and define the permuted measure
    	\[
    	\alpha^\sigma:=\alpha^{\sigma_0,\cdots,\sigma_K}= \frac{1}{m} \sum_{i=1}^m \delta_{(\widehat{x}_0^{\sigma_0(i)}, \cdots, \widehat{x}_K^{\sigma_K(i)})}.
    	\]
    	Similarly, $\alpha^\sigma$ has identical marginals to $\beta$. Applying the previous inequality yields
    	\[
    	\int_\Omega c \  {\rm d}\beta \leq \int_\Omega c \  {\rm d}\alpha^\sigma = \frac{1}{m} \sum_{i=1}^m c(\widehat{x}_0^{\sigma_0(i)}, \cdots, \widehat{x}_K^{\sigma_K(i)}).
    	\]
    	Multiplying by $m$ gives us
    	\[
    	\sum_{i=1}^m c(\omega_i) \leq \sum_{i=1}^m c(\widehat{x}_0^{\sigma_0(i)}, \cdots, \widehat{x}_K^{\sigma_K(i)}),
    	\]
    	which holds for any permutation $\sigma_0,\cdots,\sigma_K$. Therefore, $S$ is a $c$-cyclically monotone set. Since $S \subset {\rm supp}(\pi^*)$ was an arbitrary finite subset, we conclude that ${\rm supp}(\pi^*)$ is a $c$-cyclically monotone set.
    \end{proof}
    \begin{lemma}
    	\label{none}
    	Let $\{\pi^*,\{\mu_k^*\}_{k=1}^{K-1}\}$ be a solution of (\ref{KS}), set $h(\widehat{x}_0):=\inf\{c(\omega) \mid \omega\in {\rm supp}(\pi^*),\ T_0(\omega)=\widehat{x}_0\}$. If the cost function $c$ is strongly coercive, then $S(\widehat{x}_0):=\{\omega\in {\rm supp}(\pi^*) \mid T_0(\omega)=\widehat{x}_0,\ c(\omega)=h(x_0)\}$ is  a non-empty compact set.
    \end{lemma}
   \begin{proof}
   For fix $\mu_0$-a.e. $\widehat{x}_0 \in X_0$, by defined of $h(\widehat{x}_0)$, there exists $\{\omega_n\}_{n=1}^\infty \subset {\rm supp}(\pi^*)$ such that 
   \[
   T_0(\omega_n)=\widehat{x}_0 \quad\text{and}\quad \lim_{n\to \infty}c(\omega_n)=h(\widehat{x}_0).
   \]
   Since $c$ is strongly coercive, then there exsts $\varepsilon>0$ such that $K:=\{\omega \in {\rm supp}(\pi^*) \mid T_0(\omega)=\widehat{x}_0, c(\omega) \leq h(\widehat{x}_0) +\varepsilon\}$ is a compact set.
   Then,  there exists $N \in \mathbb{N}$ such that for any $n>N$, 
   \[
   \left| c(\omega_n) -h(\widehat{x}_0) \right| <\frac{\varepsilon}{3}.
    \]
    Since $\{\omega_n\}_{n\geq N}^\infty \subset K$ and $K$ is compact, there exists $\omega_{n_k}$ such that $\omega_{n_k} \to \omega^*\in K$ as $k \to \infty$.  For each $k$, $T_0(\omega_{n_k})=\widehat{x}_0$ and $T_0$ is continuous map, then 
     \[
     T_0(\omega^*)= \lim_{n \to \infty}T_0(\omega_{n_k}) =\widehat{x}_0.
     \]
     Since $c$ is lower semicontinuous, 
     \[
     c(\omega^*) \leq \liminf_{k\to\infty} c(\omega_{n_k}) = h(\widehat{x}_0).
     \]
     By defined of $h(\widehat{x}_0)$,  $h(\widehat{x}_0) \leq c(\omega^*)$, then $c(\omega^*)=h(\widehat{x}_0)$. Hence $c(\omega^*) \in S(\widehat{x}_0)$ and $S(\widehat{x}_0)$ is non-empty. Let $\{\omega_i\}_{i=1}^\infty \subset S(\widehat{x}_0)$ be a convergence  sequence and $\omega_i \to \widehat{\omega} \in {\rm supp}(\pi^*)$ as $i \to \infty$. Since $  c(\widehat{\omega}) \leq \liminf_{i\to\infty} c(\omega_{i}) = h(\widehat{x}_0)$ and $c(\widehat{\omega})\geq h(\widehat{x}_0)$, then $c(\widehat{\omega})=h(\widehat{x}_0)$. Therefore, $\widehat{\omega} \in S(\widehat{x}_0)$ and $S(\widehat{x}_0)$ is a closed set. Since $S(\widehat{x}_0) \subset K$ and $K$ is compact set, then $S(\widehat{x}_0)$ is compact set.
   \end{proof}
   \begin{lemma}
   	\label{hlsc}
   	If $c$ is locally Lipschitz continuous with respect to the initial variable  $\widehat{x}_0$, then $h(\widehat{x}_0)$ $c$ is also locally Lipschitz continuous with respect to the initial variable  $\widehat{x}_0$.
   \end{lemma}
   \begin{proof}
   	Since $c$ is locally Lipschitz continuous with respect to the initial variable  $\widehat{x}_0$, for any $\widehat{x}_0 \in \Omega$, there exists a Lipschitz constant $L>0$ and a neighborhood $U(\widehat{x}_0)$  of $\widehat{x}_0$  such that  for all $\widehat{y}_0 \in U(\widehat{x}_0)$ and all $(\widehat{x}_1,\cdots,\widehat{x}_K) \in \widehat{X}_1 \times \cdots\times \widehat{X}_K$,
   	\[
   	\left| c(\widehat{x}_0,\widehat{x}_1,\cdots,\widehat{x}_K) -c(\widehat{y}_0,\widehat{x}_1,\cdots,\widehat{x}_K)   \right| \leq L\cdot{\sf d}_0(\widehat{x}_0,\widehat{y}_0).
   	\]
   	For any $\varepsilon_1 >0$, there exists  a path  $\omega^y = (\widehat{y}_0,\widehat{y}_1,\cdots,\widehat{y}_K) \in {\rm supp}(\pi^*)$ such that 
   	\[
   	c(\omega^y)\leq h(\widehat{y}_0)+\varepsilon_1.
   	\]
   	Applying the Lipschitz condition, we have
   	\[
   	c(\widehat{x}_0, \widehat{y}_1,\cdots,\widehat{y}_K) \leq c(\omega^y) +L\cdot {\sf d}_0(\widehat{x}_0,\widehat{y}_0).
   	\]
   	Then
   	\[
   	\begin{aligned}
   		h(\widehat{x}_0) &\leq c(\widehat{x}_0, \widehat{y}_1,\cdots,\widehat{y}_K)\\
   		&\leq c(\omega^y) +L\cdot {\sf d}_0(\widehat{x}_0,\widehat{y}_0)\\
   		&\leq h(\widehat{y}_0)+\varepsilon_1 +L\cdot {\sf d}_0(\widehat{x}_0,\widehat{y}_0),
   	\end{aligned}
   	\]
   	and hence $h(\widehat{x}_0) - h(\widehat{y}_0) \leq \varepsilon_1 +L\cdot {\sf d}_0(\widehat{x}_0,\widehat{y}_0)$. For any $\varepsilon_2>0$, there exists  a path  $\omega^x = (\widehat{x}_0,\widehat{x}'_1,\cdots,\widehat{x}'_K) \in {\rm supp}(\pi^*)$ such that 
   	\[
   	c(\omega^x)\leq h(\widehat{x}_0)+\varepsilon_2.
   	\]
   	Applying the Lipschitz condition, we have
   	\[
   	c(\widehat{y}_0, \widehat{x}'_1,\cdots,\widehat{x}'_K) \leq c(\omega^x) +L\cdot {\sf d}_0(\widehat{x}_0,\widehat{y}_0).
   	\]
   	Then
   	\[
   	\begin{aligned}
   		h(\widehat{y}_0) &\leq c(\widehat{y}_0, \widehat{x}'_1,\cdots,\widehat{x}'_K)\\
   		&\leq c(\omega^x) +L\cdot {\sf d}_0(\widehat{x}_0,\widehat{y}_0)\\
   		&\leq h(\widehat{x}_0)+\varepsilon_2 +L\cdot {\sf d}_0(\widehat{x}_0,\widehat{y}_0),
   	\end{aligned}
   	\]
   	and hence $h(\widehat{y}_0) - h(\widehat{x}_0) \leq \varepsilon_2 +L\cdot {\sf d}_0(\widehat{x}_0,\widehat{y}_0)$. Since $\varepsilon_1$ and $\varepsilon_2$ are arbitrary, we conclude that
   	\[
   	\left|h(\widehat{x}_0) -h(\widehat{y}_0)  \right| \leq L\cdot {\sf d}_0(\widehat{x}_0,\widehat{y}_0).
   	\]
   	This shows that $h$ is Lipschitz continuous with respect $\widehat{x}_0 \in \mathcal{M}_0$.
   \end{proof}
    \begin{dy}(Locally control condition)
    	\label{lcc}
    	The cost function $c$ satisfies the \text{locally control condition} if for $\mu_0$-a.e. $\widehat{x}_0 \in X_0$, there exist the same sequence $\{\widehat{x}_0^n\}_{n=1}^\infty \subset X_0$ from Definition~\ref{sd} converging to $\widehat{x}_0$ and for every $\omega=(\widehat{x}_0,\cdots,\widehat{x}_K) \in S(\widehat{x}_0)$,
    	\[
    	c(\widehat{x}_0^n,\widehat{x}_1,\cdots,\widehat{x}_K) -h(\widehat{x}_0^n) = o({\rm d}_0(\widehat{x}_0^n,x_0)).
    	\]
    \end{dy}
    \begin{theorem}
    	\label{MP}
    	Assume the lower semicontinuous cost function $c$ is locally control (discrete-type) sequential twisted and strongly coercive. If $M<\infty$, then there exists a unique Borel map $T^*: X_0 \to \Omega$ such that 
    	\begin{itemize}
    		\item[${\rm (1)}$] The coupling $\pi^*=T^*_\# \mu_0$ is the unique solution of (\ref{KS}).
    		\item[${\rm (2)}$] For  $\widehat{x}_0 \in X_0$, $T^*$ decomposes as $T^*(\widehat{x}_0)=(\widehat{x}_0,T_1^*(\widehat{x}_0),\cdots,T^*_K(\widehat{x}_0))$, where
    		\[
    		T^*_k(\widehat{x}_0):=
    		\begin{cases}
    			y^*_k&\text{if}\quad k \in I(T^*(\widehat{x}_0)),\\
    			\partial_k & \text{ otherwise},
    		\end{cases}
    		\]
    		and $y_k^* \in X_k$ is the k-th component of the path $T^*(\widehat{x}_0)$.
    	\end{itemize}
    \end{theorem}
    \begin{proof}[Proof of Theorem~\ref{MP}]
    		Let $\{\pi^*,\{\mu_k^*\}_{k=1}^{K-1}\}$ be a solution of (\ref{KS}). Since $c$ is sequential twisted, for $\omega=(\widehat{x}_0,\cdots,\widehat{x}_K) \in \Omega$, there exists sequence $\{\widehat{x}_0^n\}_{n=1}^\infty \subset X_0$ converging to $\widehat{x}_0$ and  the limit
    		\[
    		D_{\{\widehat{x}_0^n\},\,\widehat{x}_0} c(\omega) := \lim_{n \to \infty} \frac{c(\widehat{x}_0^n,\widehat{x}_1,\cdots,\widehat{x}_K) - c(\widehat{x}_0,\widehat{x}_1,\cdots\,\widehat{x}_K)}{{\sf d}_0(\widehat{x}_0^n,\widehat{x}_0)}
    		\]
    		exists in $\mathbb{R}$. Since $c$ is locally control,
    		\[
    		\frac{c(\widehat{x}_0^n,\widehat{x}_1,\cdots,\widehat{x}_K) - c(\widehat{x}_0,\widehat{x}_1,\cdots\,\widehat{x}_K)}{{\sf d}_0(\widehat{x}_0^n,\widehat{x}_0)} = \frac{h(\widehat{x}_0^n)+o({\sf d}_0 (\widehat{x}_0^n,\widehat{x}_0))-h(\widehat{x}_0)}{{\sf d}_0(\widehat{x}_0^n,\widehat{x}_0)}.
    		\]
    		Therefore,
    		\[
    		\lim_{n \to \infty}\frac{c(\widehat{x}_0^n,\widehat{x}_1,\cdots,\widehat{x}_K) - c(\widehat{x}_0,\widehat{x}_1,\cdots\,\widehat{x}_K)}{{\sf d}_0(\widehat{x}_0^n,\widehat{x}_0)} = \lim_{n \to \infty}\frac{h(\widehat{x}_0^n) -h(\widehat{x}_0)}{{\sf d}_0(\widehat{x}_0^n,\widehat{x}_0)},
    		\]
    		and $D_{\{\widehat{x}_0^n\} ,\,\widehat{x}_0}c(\omega)=D_{\{\widehat{x}_0^n\} ,\,\widehat{x}_0}h$ for any $\omega \in S(\widehat{x}_0)$. Suppose there exists $\omega_1\in S(\widehat{x}_0)$ and $\omega_2 \in S(\widehat{x}_0)$ such that $\omega_1\neq \omega_2$. Since $\omega_1$ and $\omega_2$ are not equal, by $c$ is sequential twisted, $\displaystyle D_{\{\widehat{x}_0^n\} ,\,\widehat{x}_0}c(\omega_1) \neq D_{\{\widehat{x}_0^n\} ,\,\widehat{x}_0}c(\omega_2)$. But $\omega_1 ,\omega_2 \in S(\widehat{x}_0)$,
    		\[
    		D_{\{\widehat{x}_0^n\} ,\,\widehat{x}_0}c(\omega_1)= D_{\{\widehat{x}_0^n\} ,\,\widehat{x}_0}h=D_{\{\widehat{x}_0^n\} ,\,\widehat{x}_0}c(\omega_2),
    		\]
    		contradiction. Then $S(\widehat{x}_0)$ is  a singleton. In particular, if Polish spaces $X_0$ that lack a nontrivial curve structure, $S(\widehat{x}_0)=\mathcal{F}(\widehat{x}_0)$ is also a singleton. For each $\widehat{x}_0 \in X_0$, define $T^*(\widehat{x}_0)$ as the unique element of $S(\widehat{x}_0)$. Consider
    		\[
    		\begin{aligned}
    			G_r(T^*):&=\{(\widehat{x}_0,\omega) \in X_0 \times \Omega \mid \omega \in S(\widehat{x}_0) \} \\
    			&=\{(\widehat{x}_0,\omega) \in X_0 \times \Omega \mid \omega \in  {\rm supp}(\pi^*),\,T_0(\omega)=\widehat{x}_0\} \cap \{ (\widehat{x}_0,\omega) \in X_0 \times \Omega \mid c(\omega) =h(\widehat{x}_0)\}.
    		\end{aligned}
    		\]
    		Since $\{(\widehat{x}_0,\omega) \in X_0 \times \Omega \mid \omega \in  {\rm supp}(\pi^*),\,T_0(\omega)=\widehat{x}_0\} $ is a closed set and $ \{ (\widehat{x}_0,\omega) \in X_0 \times \Omega \mid c(\omega) =h(\widehat{x}_0)\}$ is a Borel set, then  $G_r$ is a Borel set and hence $T^*$ is a Borel map (see  \cite{MR188994}). By the disintegration theorem, for any Borel set $A \subset \Omega$
    		\[
    		\pi^*(A)=\int_{X_0} [\delta_{\widehat{x}_0} \otimes \kappa_{\widehat{x}_0}](A)\ \mu_0({\rm d}\widehat{x}_0),
    		\]
    		where $\delta_{x_0}$ is a Dirac measure and $\kappa_{\widehat{x}_0}$ is the conditional probability measure of $\omega$  given $\widehat{x}_0$. Since $\kappa_{\widehat{x}_0}(S(\widehat{x}_0) )=1$  and $S(\widehat{x}_0) $ is a singleton,  then $\kappa_{\widehat{x}_0}=\delta_{T^*(\widehat{x}_0)}$. Therefore
    		\[
    		\pi^*(A)=\int_{X_0} [\delta_{\widehat{x}_0} \otimes \kappa_{\widehat{x}_0}](A)\ \mu_0({\rm d}\widehat{x}_0)=T^*_\# \mu_0(A).
    		\]
    		Let $\{\pi',\{\mu'_k\}_{k=1}^{K-1}\}$ be another optimal solution of (\ref{KS}). By the disintegration theorem, for any Borel set $A \subset \Omega$
    		\[
    		\pi'(A)=\int_{X_0} [\delta_{\widehat{x}_0} \otimes \kappa'_{\widehat{x}_0}](A)\ \mu_0({\rm d} \widehat{x}_0). 
    		\]
    		Similarly, $S(\widehat{x}_0) $ is also a singleton and $\kappa_{\widehat{x}_0}(S(\widehat{x}_0) )=1$, then $\kappa'_{\widehat{x}_0}=\delta_{T^*(\widehat{x}_0)}=\kappa_{\widehat{x}_0}$. Hence,
    		\[
    		\pi'(A)=\int_{X_0} [\delta_{\widehat{x}_0} \otimes \kappa'_{\widehat{x}_0}](A)\ \mu_0({\rm d} \widehat{x}_0)=\int_{X_0} [\delta_{\widehat{x}_0} \otimes \kappa_{\widehat{x}_0}](A)\ \mu_0({\rm d} \widehat{x}_0)=\pi^*(A).
    		\] 		
    		So we have proved that Theorem~\ref{MP}.
    \end{proof}
    \begin{remark}
    	The local control condition can be formulated either with the minimal cost function $h(\widehat{x}_0)$ or, alternatively, with the Kantorovich duality functions $v_0$ as in the proof of Theorem~\ref{MP}
    	\[
    	c(\widehat{x}_0^n,\widehat{x}_1,\cdots,\widehat{x}_K)-v_0(\widehat{x}_0^n)=o({\rm d}_0(\widehat{x}_0^n,x_0)).
    	\]
    	These two formulations are not equivalent in general, yet each of them, under suitable regularity assumptions, suffices to guarantee the uniqueness of the Monge solution.
    \end{remark}
    For $i=0,\cdots,K-1$, if $\{\pi^*,\{\mu_k^*\}_{k=1}^{K-1}\}$ is the unique minimizer of the optimization problem (\ref{KS}), we define the minimal cost function as
    \[
    \widehat{c}_{i,i+1}(\widehat{x}_i, \widehat{x}_{i+1}) :=  
    \begin{cases} 
    	c_{i,i+1}(x_i, x_{i+1}) &\text{if}\quad \widehat{x}_i \in X_i \quad\text{and}\quad\widehat{x}_{i+1} \in X_{i+1}, \\
    	c_{i,k_j}(x_i, x_{k^i_j}) &\text{if}\quad\widehat{x}_i \in X_i \quad\text{and}\quad \widehat{x}_{i+1} = \partial_{i+1}, \\
    	0&\text{otherwise},
    \end{cases}
    \]
    where $k^i_j \in I(T^*( {T^*_i}_\#(\widehat{x}_i)))$ satisfies $\displaystyle k^i_j=\arg\min_{\substack{i+1<j\leq K\\\widehat{x}_j \in X_j}} I(T^*( {T^*_i}_\#(\widehat{x}_i)))$. Therefore, $\displaystyle c(\omega)=\sum_{i=0}^{K-1}\widehat{c}_{i,i+1}(\widehat{x}_i,\widehat{x}_{i+1})$ for  $\pi^*$-almost every $\omega$.
    \begin{theorem}
    	\label{emp}
    	Fix $\mu_0\in\mathcal{P}(X_0)$ and $\mu_K \in \mathcal{P}(X_K)$. If $\{\pi^*,\{\mu_k^*\}_{k=1}^{K-1}\}$ is the unique minimizer of the optimization problem (2.1), then
    	\[
    	M=\sum_{i=0}^{K-1}\widehat{C}_{i,i+1}(\mu^*_i,\mu^*_{i+1}),
    	\]
    	where $\displaystyle\widehat{C}_{i,i+1}(\mu_i^*,\mu_{i+1}^*):=\inf_{\pi_i\in\Pi(\mu_i^*,\mu_{i+1}^*)}\int_{\widehat{X}_i \times \widehat{X}_{i+1}}\widehat{c}_{i,i+1}(\widehat{x}_i,\widehat{x}_{i+1}) \pi_i({\rm d}\widehat{x}_i,{\rm d}\widehat{x}_{i+1})$.
    \end{theorem}
    \begin{proof}
    	Let  $\pi^* \in \Pi ( \mu_0,\mu_1^* , \cdots, \mu_{K-1}^*,\mu_K )$ be a  unique minimizer of the optimization problem (\ref{KS}), satisfying
    	\[
    	(T_k)_\# \pi^*=\mu_k^* \quad\text{for every}\quad k=0,\cdots,K,
    	\]
    	where $\mu_0^*:=\mu_0$ and $\mu_K^*:=\mu_K$. For $\omega=(\widehat{x}_0,\cdots,\widehat{x}_K) \in \Omega$, we define the projection map $(T_i,T_{i+1}) : \Omega \to \widehat{X}_i \times \widehat{X}_{i+1}$ as
    	\[
    	(T_i,T_{i+1})(\omega):= (\widehat{x}_i,\widehat{x}_{i+1}).
    	\]
    	Then, the projection coupling $\pi_i^\prime$ defined as a pushforward measure 
    	\[
    	\pi_i^\prime:=\left(T_i,T_{i+1}\right)_\#\pi^*.
    	\]
    	Since the projection map  $(T_i,T_{i+1})$ is continuous (as a projection in a product space), this is well-defined. For every Borel sets $A \subset \widehat{X}_i$ and $B \subset\widehat{X}_{i+1}$,
    	\[
    	\begin{aligned}
    		&\pi_i^\prime(A \times \widehat{X}_{i+1})=\pi^*(\{\omega\mid \widehat{x}_i \in A \})={T_i}_\# \pi^*(A)=\mu_i^*(A),\\
    		&\pi_i^\prime(\widehat{X}_i \times B)=\pi^*(\{\omega\mid \widehat{x}_{i+1} \in B \})={T_{i+1}}_\# \pi^*(B)=\mu_{i+1}^*(B).
    	\end{aligned}
    	\]
    	Thus, $\pi^\prime_i \in \Pi (\mu_i^*,\mu_{i+1}^* )$. For each  $i=0,\cdots,K-1$ and $\varepsilon_i>0$, there exists $\pi_i^* \in \Pi(\mu_i^*,\mu_{i+1}^*)$ such that
    	\[
    	\int_{\widehat{X}_i \times \widehat{X}_{i+1}} \widehat{c}_{i,i+1}\ {\rm d}\pi_i^* \leq \widehat{C}_{i,i+1}(\mu^*_i,\mu^*_{i+1})+\varepsilon_i.
    	\]
    	By the construction used in the proof of Theorem~\ref{Mfinity}, for the given couplings $\{\pi_i^*\}_{i=0}^{K-1}$ where $\pi_i^*\in \Pi (\mu_i^*,\mu_{i+1}^* )$, there exists a global coupling $\pi \in \Pi(\mu_0,\mu_1^*,\cdots,\mu_{K-1}^*,\mu_K )$ such that $(T_i,T_{i+1})_\#\pi=\pi_i^*$ for $i=0,\cdots,K-1$ and ${T_k}_\#\pi=\mu_k^*$ for $ k=0,\cdots,K$. Then 
    	\[
    	\begin{aligned}
    		\int_{\Omega} c(\omega)\ \pi^*({\rm d}\omega)&\leq \int_{\Omega} c(\omega)\ \pi({\rm d}\omega)\\
    		&=\sum_{i=0}^{K-1}\int_{\widehat{X}_i \times \widehat{X}_{i+1}} \widehat{c}_{i,i+1}(\widehat{x}_i,\widehat{x}_{i+1}) \ \pi^*_i({\rm d}\widehat{x}_i,{\rm d}\widehat{x}_{i+1})\\
    		&\leq\sum_{i=0}^{K-1}\left( \widehat{C}_{i,i+1}(\mu_i^*,\mu_{i+1}^*)+\varepsilon_i \right).
    	\end{aligned}
    	\]
    	Therefore $\displaystyle M \leq \sum_{i=0}^{K-1}\left( \widehat{C}_{i,i+1}(\mu^*_i,\mu^*_{i+1})+\varepsilon_i \right)$. On the other hand, 
    	\[
    	\begin{aligned}
    		\int_{\Omega} c \ {\rm d}\pi^*
    		&=\sum_{i=0}^{K-1}\int_{\widehat{X}_i \times \widehat{X}_{i+1}} \widehat{c}_{i,i+1}\ {\rm d}\pi_i^\prime \geq \sum_{i=0}^{K-1} \widehat{C}_{i,i+1}(\mu_i^*,\mu_{i+1}^*).
    	\end{aligned}
    	\]
    	Thus, $\displaystyle\sum_{i=0}^{K=1} \widehat{C}_{i,i+1}(\mu_i^*,\mu_{i+1}^*) \leq M \leq \sum_{i=0}^{K-1}\left( \widehat{C}_{i,i+1}(\mu_i^*,\mu_{i+1}^*)+\varepsilon_i \right)$. Letting each $\varepsilon_{i} \to 0$, we obtain $\displaystyle M = \sum_{i=0}^{K-1} \widehat{C}_{i,i+1}(\mu^*_i,\mu^*_{i+1} ) $.
    \end{proof}
    
    \section{HJMOT on Smooth Manifolds}
    \label{sec4}
    The HJMOT framework can be applied naturally to Riemannian geometry. Consider complete boundaryless Riemannian manifolds $(\mathcal{M}_0, g_0)$ and $(\mathcal{M}_K, g_K)$ as source and target spaces. Intermediate spaces are augmented as $\widehat{\mathcal{M}}_k := \mathcal{M}_k \cup \{\partial_k\}$ for $k = 1, \cdots, K-1$, where each $\mathcal{M}_k$ is a complete Riemannian manifolds and $\partial_k$ is a topologically isolated point. The path space $\displaystyle\Omega := \mathcal{M}_0 \times \left( \prod_{k=1}^{K-1} \widehat{\mathcal{M}}_k \right) \times \mathcal{M}_K$ is endowed with the product topology. Given a family of lower semicontinuous pairwise cost functions 
    \[
    \{c_{i,j} :\mathcal{M}_i \times \mathcal{M}_j \to [0,+\infty] \  | \ 0\leq i<j\leq K   \},
    \]
    we define the path cost as
    \[
    c(\omega):=\sum_{l=0}^{n-2}c_{k_l,k_{l+1}}(x_{k_l},x_{k_{l+1}})\quad\text{for}\quad\Psi(\omega)=(x_{k_0},
    	\cdots,x_{k_{n-1}}),
    \]
    where $\Psi(\omega)$ denotes the sequence of visited points extracted from $\omega$ (as defined in Section~\ref{sec2}).
    \begin{theorem}[Kantorovich solution]
    	\label{RKP}
    	We fix $\mu_0\in\mathcal{P}(\mathcal{M}_0)$ and $\mu_K \in \mathcal{P}(\mathcal{M}_K)$. We seek to solve the following  optimization problem
    	\begin{equation}
    		\label{RK}
    		M:=\inf_{ \substack{ \mu_k\in \mathcal{P}(\widehat{\mathcal{M}}_k)\text{ for } k=1,\cdots,\,K-1 \\ \pi \in \Pi(\mu_0,\cdots,\,\mu_K)} } \int_{\Omega} c\ {\rm d}\pi,
    	\end{equation}
    	where $\Pi(\mu_0,\cdots,\mu_K)=\{ \pi \in \mathcal{P}(\Omega) \mid {T_k}_\# \pi = \mu_k \text{ for  }k\in\{0,\cdots,K\}  \}$.  If $M$ in (4.1) is finite and $c$ is coercive, then the optimization problem (4.1) has minimizer $\{\pi^*,\{\mu_k^*\}_{k=1}^{K-1}\}$.
    \end{theorem}
    \begin{remark}
    	Analogous to Theorem~\ref{Mfinity}, the finiteness of $M$ in (\ref{RKP}) is guaranteed under the condition that there exist intermediate measures $\mu_k \in \mathcal{P}(\widehat{\mathcal{M}}_k)$ for $k=1,\cdots,K-1$ such that $\displaystyle\sum_{i=0}^{K-1} \tilde{C}_{i,i+1}(\mu_i,\mu_{i+1}) < \infty$.
    \end{remark}
    To address directional derivatives, fix a path $\omega = (\widehat{x}_0, \cdots, \widehat{x}_K)$ and let $\gamma(t) = \exp_{\widehat{x}_0}(tv)$ denote the geodesic satisfying $\gamma(0) = \widehat{x}_0$ and $\dot{\gamma}(0) = v$ for $t\in[0,1]$.
    \begin{dy}[Differentiability along geodesics]
    	\label{DAG}
    	For a fixed path $\omega=(\widehat{x}_0,\cdots,\widehat{x}_K) \in \Omega$, we say that the 
    	lower semicontinuous cost function $c(\omega)$ is \textit{differentiability along geodesics} at $\widehat{x}_0$ if for any $v \in T_{\widehat{x}_0} \mathcal{M}_0$, the limit 
    	\[
    	D_{v,\,,\widehat{x}} c(\omega):=\left.\frac{\mathrm{d}}{\mathrm{d}t}\right|_{t=0} 
    	c(\exp_{x_0}(t v),\,x_1, \dots, x_K)
    	\] 
    	always exists in $\mathbb{R}$ and uniquely determined by $v$ (i.e. independent 
    	of the particular curve chosen to approach $x_0$ in direction $v$).
    \end{dy}
    \begin{dy}[Geodesic twistedness condition]
    	\label{GTC}
    	A  cost function $c$ is called \textit{geodesic twisted} if $c$ is differentiability along geodesics at $\widehat{x}_0$ and the mapping
    	\[
    	(\widehat{x}_1,\cdots,\widehat{x}_K) \mapsto D_{v,\,\widehat{x}_0} c(\widehat{x}_0,\widehat{x}_1,\cdots,\widehat{x}_K)
    	\]
    	is injective on the set $S_{\widehat{x}_0}:=\{\omega=(\widehat{x}_0,\widehat{x}_1,\cdots,\widehat{x}_K) \in S \mid  \widehat{x}_i \in \widehat{X}_i \text{ for }1 \leq i \leq K\}$ , for every $v \in T_{\widehat{x}_0}\mathcal{M}_0$,  where $S$ is a $c$-cyclically monotone set.
    \end{dy}
    \begin{dy}[Locally control condition]
    	\label{Rlcc}
    	The cost function $c: \Omega \to [0, +\infty]$ is said to satisfy the \textit{locally control condition} if for $\mu_0$-a.e. $\widehat{x}_0 \in \mathcal{M}_0$, for every $v \in T_{\widehat{x}_0}\mathcal{M}_0$ and every $\omega=(\widehat{x}_0,\cdots,\widehat{x}_K) \in S(\widehat{x}_0)$, 
    	\[
    	c({\rm exp}_{\widehat{x}_0}(tv),\widehat{x}_1,\cdots,\widehat{x}_K ) - h({\rm exp}_{\widehat{x}_0}(tv))  = o(t).
    	\]
    	Here $h(\widehat{x}_0) := \inf\{c(\omega) \mid \omega \in {\rm supp}(\pi^*),\ T_0(\omega) = \widehat{x}_0\}$ and $S(\widehat{x}_0) := \{\omega \in {\rm supp}(\pi^*)\mid T_0(\omega) = \widehat{x}_0,\ c(\omega) = h(\widehat{x}_0)\}$
    \end{dy}
    \begin{lemma}
    	\label{edy}
    	Let $(\mathcal{M}_0, g_0)$ be a smooth Riemannian manifold, and let ${\sf d}_0$ be the Riemannian distance function induced by $g_0$. Assume the cost function $c$ on $\Omega$ is differentiability along geodesics $\widehat{x}_0 \in \mathcal{M}_0$ in the sense of Definition~\ref{DAG}. Then $c$ is sequential differentiable at $\widehat{x}_0$ in the sense of Definition~\ref{sd}.
    \end{lemma}
    \begin{proof}
    	Assume the statement of Definition~\ref{DAG} holds. Then for  $v \in T_{\widehat{x}_0} \mathcal{M}_0$ such that the limit
    	\[
    	D_{v,\,\widehat{x}_0} c(\omega) := \left.\frac{\mathrm{d}}{\mathrm{d}t}\right|_{t=0} 
    	c(\exp_{\widehat{x}_0}(t v),\,\widehat{x}_1, \cdots, \widehat{x}_K)
    	\]
    	exists in $\mathbb{R}$ and is uniquely determined by $v$. Now consider the sequence $\{\widehat{x}_0^n\}$ defined by $\widehat{x}_0^n = \exp_{\widehat{x}_0}(t_n v)$ where $t_n \to 0^+$. Since the exponential map is smooth and ${\sf d}_0(\widehat{x}_0^n, \widehat{x}_0) = t_n \|v\|_{g_0}$ for sufficiently small $t_n$, we have
    	\[
    	\begin{aligned}
    		D_{\{\widehat{x}_0^n\} ,\,\widehat{x}_0}c(\omega)&=\lim_{n \to \infty} \frac{c(\widehat{x}_0^n, \widehat{x}_1, \cdots, \widehat{x}_K) - c(\widehat{x}_0, \widehat{x}_1, \cdots, \widehat{x}_K)}{{\sf d}_0(\widehat{x}_0^n, \widehat{x}_0)} \\
    		&= \frac{1}{\|v\|_{g_0}} \lim_{n \to \infty} \frac{c(\exp_{\widehat{x}_0}(t_n v), \widehat{x}_1, \cdots, \widehat{x}_K) - c(\widehat{x}_0, \widehat{x}_1, \cdots, \widehat{x}_K)}{t_n} \\
    		&= \frac{1}{\|v\|_{g_0}} D_{v,\,\widehat{x}_0} c(\omega).
    	\end{aligned}
    	\]
    	Therefore, the limit exists and equals $\frac{1}{\|v\|_{g_0}} D_{v,\,\widehat{x}_0} c(\omega)$. This shows that $c$ is sequential differentiable at $\widehat{x}_0$ in the sense of Definition~\ref{sd} along the sequence $\{\widehat{x}_0^n\}$.
    \end{proof}
    \begin{theorem}[Monge solution]
    	\label{RMP}
    	Assume the lower semicontinuous cost function $c$ on $\Omega$ is  locally control  geodesic twisted and strongly  coercive. If $M < \infty$, then there exists a unique Borel map $T^*: \mathcal{M}_0 \to \Omega$ such that 
    	\begin{itemize}
    		\item[${\rm (1)}$] The coupling $\pi^* = {T^*}_\# \mu_0$ is the unique solution of (\ref{RK}). 
    		\item[${\rm (2)}$] For $\mu_0$-a.e. $\widehat{x}_0 \in \mathcal{M}_0$, $T^*$ decomposes as 
    		$T^*(\widehat{x}_0)=(\widehat{x}_0, T_1^*(\widehat{x}_0), \cdots, T_K^*(\widehat{x}_0))$, where 
    		\[
    		T^*_k(\widehat{x}_0) := 
    		\begin{cases}
    			y^*_k, & \text{if}\quad k \in I(T^*(\widehat{x}_0)),\\
    			\partial_k, & \text{otherwise},
    		\end{cases}
    		\] 
    		and $y^*_k \in \mathcal{M}_k$ is the $k$-th component of the path $T^*(\widehat{x}_0)$.
    	\end{itemize}
    \end{theorem}
    \begin{theorem}
    	\label{eRMP}
    Assume the cost function $c$ is strongly coercive and locally Lipschitz with respect to the first variable and $\mu_0 \ll \mathcal{L}$. If for $\mu_0$-almost every $\widehat{x}_0 \in \mathcal{M}_0$, the  mapping $(\widehat{x}_1,\cdots,\widehat{x}_K) \mapsto D_{v,\,\widehat{x}_0}c(\omega)$ is continuous and injective on $S(\widehat{x}_0)$ for every $v \in T_{\widehat{x}_0}\mathcal{M}_0$. Then there has unique Monge solution of (\ref{RK}) and $D_{v,\,\widehat{x}_0}c(\omega)=D_{v,\,\widehat{x}_0}v_0(\widehat{x}_0)$ where $v_k$ is are Kantorovich duality functions.
    \end{theorem}
    \begin{proof}
    	SInce $c$ is  locally Lipschitz with respect to the first variable, clearly $v_0$ is also locally Lipschitz with respect to the first variable. Since ${\rm supp}(\pi^*)$ is closed set and separable, there exists a countable dense subset $D=\{\omega^1,\omega^2,\cdots\}\subset {\rm supp}(\pi^*)$. For each fixed $\omega^i=(\widehat{x}_0^i,\widehat{x}_1^i,\cdots,\widehat{x}_K^i)  \in D$, we define the function $f^{\omega^i}(y_0):=c(y_0,\widehat{x}_1^i,\cdots,\widehat{x}_K^i) -v_0(y_0)-\sum_{k=1}^K v_k(\widehat{x}_k)$ for $y_0 \in \mathcal{M}_0$. For each fixed $\omega^i=(\widehat{x}_0,\widehat{x}_1^i,\cdots,\widehat{x}_K^i) \in D$,  by Rademacher's theorem (\cite[Theorem 3.1]{MR2177410}), there exists a set $N_{\omega^i}\subset \mathcal{M}_0$ with $\mu_0(N_{\omega^i})=0$ such that for any $y_0 \in \mathcal{M}_0 \setminus N_{\omega^i}$, $f^{\omega^i}$ is differentiable at $y_0$, i.e. for any $v \in T_{\widehat{x}_0}\mathcal{M}_0$, $f^{\omega^i}({\rm exp}_{\widehat{x}_0}(tv))=f^{\omega^i}(\widehat{x}_0) + D_{v,\,\widehat{x}_0} f^{\omega^i} +o(t)$. If $\widehat{x}_0 \in T_0({\rm supp}(\pi^*)) \setminus N_{\omega^i}$, since $f^{\omega^i}(y_0)\geq 0$ for any $y_0 \in \mathcal{M}_0$ and $f^{\omega^i}(\widehat{x}_0)=0$, then  $D_{v,\,\widehat{x}_0} f^{\omega^i}=0$. Therefore, for any $v \in T_{\widehat{x}_0}\mathcal{M}_0$,
    	\[
    	f^{\omega^i}({\rm exp}_{\widehat{x}_0}(tv))=c({\rm exp}_{\widehat{x}_0}(tv),\widehat{x}_1^i,\cdots,\widehat{x}_K^i) - v_0({\rm exp}_{\widehat{x}_0}(tv))=o(t).
    	\]
    	Put $\displaystyle N_D:=\bigcup_{i=0}^\infty N_{\omega^i}$, since each $N_{\omega^i}$ has $\mu_0$-measure zero and the union is countable, we have $\mu_0(N_D)=0$. For every $ \widehat{x}_0 \in T_0({\rm supp}(\pi^*)) \setminus N_D$, $v \in T_{\widehat{x}_0} \mathcal{M}_0$ and every $\omega^i =(\widehat{x}_0^i,\widehat{x}_1^i,\cdots,\widehat{x}_K^i)\in D$, we have
    	\[
    	c({\rm exp}_{\widehat{x}_0}(tv),\widehat{x}_1^i,\cdots,\widehat{x}_K^i) - v_0({\rm exp}_{\widehat{x}_0}(tv))=o(t).
    	\]
    	Since $D$ is dense in ${\rm supp}(\pi^*)$, for any $\omega \in S(\widehat{x}_0)$, there exists $\omega^n \in D$ such that $\omega^n \to \omega$ as $n \to \infty$. Therefore,
    	\[
    	    \lim_{n\to \infty} D_{v,\,\widehat{x}_0} c(\omega^n)= D_{v,\,\widehat{x}_0} c(\omega).
    	\]
    	Since for each $\omega^n \in D$, we have $ D_{v,\,\widehat{x}_0} c(\omega^n)=  D_{v,\,\widehat{x}_0} v_0(\widehat{x}_0)$, then $ D_{v,\,\widehat{x}_0} c(\omega)=  D_{v,\,\widehat{x}_0} v_0(\widehat{x}_0)$ for every  $\omega \in S(\widehat{x}_0)$. Since $(\widehat{x}_1,\cdots,\widehat{x}_K) \mapsto D_{v,\,\widehat{x}_0} c(\omega)$ is injective on $S(\widehat{x}_0)$, then $S(\widehat{x}_0)$ is a singleton. Hence (\ref{RK}) has unique Monge solution.
    \end{proof}
    \par We note that, although in this chapter we have assumed all spaces to be smooth Riemannian manifolds for convenience, the proofs of the existence and uniqueness of Monge solutions actually only require the initial space $\mathcal{M}_0$ to be a complete smooth Riemannian manifold. The other spaces $\mathcal{M}_1, \ldots, \mathcal{M}_K$ can be taken to be complete separable metric spaces (Polish spaces), as the key differential structures (differentiability along geodesics) and the twist condition are only needed for $\mathcal{M}_0$
   \par  Chapter 4 represents a concrete realization of the abstract framework from Chapter 3 within the context of smooth Riemannian manifolds. In fact, this framework admits natural extensions to more general geometric settings, including Alexandrov spaces (\cite{MR4734965} \cite{MR2442054}), RCD spaces (\cite{MR3205729} \cite{MR2984123}   \cite{MR3381131}), and sub-Riemannian manifolds (\cite{MR2647137} \cite{MR1867362}). Notably, the verification of the local control condition in Theorem 4.3 relies on suitable generalizations of Rademacher's theorem, which are available in these broader contexts, thereby enabling a seamless extension of the HJMOT framework to these more general geometric spaces.

   \section*{Acknowledgments}
   The author would like to express his deepest gratitude to Professor Kazuhiro Kuwae, 
   for his supervision, invaluable guidance, correction, and continuous support throughout the preparation of this work. 
   The author is also grateful to Dr.~Ludovico Marini for his constructive remarks and helpful discussions, 
   and to Professors Syota Esaki and Ayato Mitsuishi for their valuable suggestions, as well as to Professor Jun Kitagawa for his valuable suggestions at the early stage of this work.
   
   \bibliographystyle{amsplain}
   \bibliography{ref.bib}

   \vfill
   \noindent Zijian Xu \
   Department of Applied Mathematics, Fukuoka University, Fukuoka 814-0180, Japan \
   Email: xuzijian823@gmail.com
   
\end{document}